\documentclass[journal]{IEEEtran}
\IEEEoverridecommandlockouts                              

\usepackage{amsmath}
\usepackage{amssymb}
\usepackage{amsthm}
\usepackage{cite}
\usepackage{paralist}
\usepackage{bm}
\usepackage{xcolor}
\usepackage{graphicx}
\usepackage{subfigure}
\usepackage{epsfig}
\graphicspath{{pic/}}

\newtheorem{C}{Corollary}

\newtheorem{R}{Remark}
\newtheorem{T}{Theorem}

\newtheorem{Lm}{Lemma}
\newtheorem{Asmp}{Assumption}

\title{\LARGE \bf
Quantifying Bounds of Model Gap for Synchronous Generators
}

\author{Peng Wang,~\IEEEmembership{Member,~IEEE,}
Shaobu Wang,~\IEEEmembership{Senior Member,~IEEE,}
Renke Huang,~\IEEEmembership{Member,~IEEE,}\\
Zhenyu Huang,~\IEEEmembership{Fellow,~IEEE}
\thanks{Financial support for this work was provided by the Advanced Scientific Computing Research (ASCR) program of the U.S. Department of Energy (DOE) Office of Science and the Advanced Grid Modeling (AGM) program of the DOE Office of Electricity. Pacific Northwest National Laboratory is operated by Battelle for the DOE under Contract DE-AC05-76RL01830.}
\thanks{Peng Wang, Shaobu Wang, and Zhenyu Huang are with Pacific Northwest National Laboratory, 902 Battelle Boulevard, Richland, WA. 99354. Corresponding emails: {\tt\small peng.wang@pnnl.gov}, {\tt\small shaobu.wang@pnnl.gov}, {\tt\small zhenyu.huang@pnnl.gov}.}%
}

\begin{document}

\maketitle
\thispagestyle{empty}
\pagestyle{empty}

\begin{abstract}
In practice, uncertainties in parameters and model structures always cause a gap between a model and the corresponding physical entity.
Hence, to evaluate the performance of a model, the bounds of this gap must be assessed.
In this paper, we propose a trajectory-sensitivity--based approach to quantify the bounds of the gap.
The trajectory sensitivity is expressed as a linear time-varying system. We thus first derive several bounds for a general linear time-varying system in different scenarios.
The derived bounds are then applied to obtain bounds of the model gap for generator plant models with different types of structural information, e.g., models of different orders. 
Case studies are carried out to show the efficacy of the bounds through synchronous generator models on different accuracy levels.
\end{abstract}
\section{Introduction}
Integrity of plant models in power systems is crucial for a reliable and robust power grid, which is key to economical electricity delivery to power consumers. Also, long-term planning and short-term operational decisions depend heavily on studies and investigations using plant models.
Imprecise models will result in conclusions that deviate from reality, which jeopardizes system designs and affects decision-making.
Model mismatch has been an important cause of a few blackouts in North America, e.g., the 1996 Western States outage and the 2011 blackout in San Diego.
In the study on these blackouts, simulations with imprecise models show that systems are stable and well-damped, but these systems eventually collapse, which contradicts the simulation study with mismatched models.
Thus, investigating the mismatch between models and plants is important for a reliable, robust, and economical grid. 
Much effort has been devoted to such investigation, to improve model quality.


Model calibration has been a hot research topic in the last decade. The basic idea is to inject measurement data, e.g., voltage magnitude and
frequency/phase angle, into the power plant terminal bus in dynamic simulation and compare the simulated response with the actual measured data, as in \cite{Kosterev2004hydro, Huang2013PMU}.
Such a method is called "event playback" and has been included in major software vendor tools, e.g., GE PSLF, Siemens PTI PSS/E, Powertech Labs TSAT, and PowerWorld Simulator \cite{VendorTool}.
Once model discrepancy is observed with event playback or another method, model identification and calibration are needed.
Methods and tools to calibrate models are developed using linear or nonlinear curve fitting techniques, either for power plants as in \cite{Pourbeik2013ModelValidation,EPRI} or for general plants, as in \cite{MathWorks}, in which the methods can be applied to power plants.
But it has been reported that multiple solutions may exist for the same model performance with the methods and tools in \cite{VendorTool,EPRI,MathWorks}, making it difficult to identify the true parameters of the model.
Simultaneous perturbation stochastic approximation methods
based on particle swarm optimization approaches are developed in \cite{Nayak2016PMU, Tsai2017PMU}.
Such methods are reliable and sufficiently precise, but they are time-consuming, for they require extensive iterations of simulations to obtain good performance.
A probabilistic collocation method is used to evaluate the uncertainty in power systems with a few uncertain parameters in \cite{Hockenberry2004PCM}.
Feature-based searching algorithms are investigated in \cite{Borden2012NormalForm,Borden2013Tool,Borden2014Proj}, which help identify the problematic parameters but lack calibration accuracy.
A dynamic-state-estimation-based generator parameter identification algorithm is proposed in \cite{Huang2013Identification} with good results for the test case, but it has been shown to be very sensitive to measurement noise.
A rule-based approach is reported in \cite{Hajnoroozi2015PMU}, which classifies parameters based on their effects on simulation responses. 
But this approach requires extensive trial and error and knowledge of the physics of the system, which may not be acquired with enough precision.
A method and a tool suite are developed to validate and calibrate models based on trajectory sensitivity analysis and advanced ensemble Kalman
filters in \cite{Huang2018Calibrating}.

The latest results show that most existing approaches can handle parameter errors well but model structure errors are very difficult to be calibrated. For example, if one uses a second-order model to approximate a fourth-order model's response. No matter how to tune the parameters of the second-order model, their response curves won't match. This motivates us to study the bound of model structures (model gap). 
Such quantification will enable us to know the range of the gap in advance and be prepared to take precautions to ensure power systems run reliably, robustly, and economically.
This is important for both long-term planning and short-term operational decisions.

Quantifications of the model gap are investigated based on trajectory sensitivity in this paper.
Two bounds for model gaps are derived via analysis of trajectory sensitivity with respect to imprecise structural information.
Because imprecise information probably results from changes of parameters and unmodeled dynamics,
the first bound quantifies model gaps caused by imprecise parameters and the second one by unmodeled dynamics.
Scenarios for model gaps caused by changes of parameters and unmodeled dynamics are both simulated to show the efficacy of the derived bounds.
Unlike existing results on using trajectory sensitivity for model gaps in power systems in the literature, e.g., \cite{Hiskens2006Sensitivity}, explicit bounds of the trajectory sensitivity are obtained in this paper.

The rest of this paper is organized in a generic-to-specific way. That is, we first present the results for trajectory sensitivity for general models and then apply those general results to quantify the model gaps of synchronous generators.
Some preliminary information is introduced in Section~\ref{sec:prel}. Several bounds for linear time-varying (LTV) systems are derived in Section~\ref{sec:LTV_bound}. We apply the bounds to quantify the model gaps of synchronous generator models on different accuracy levels in Section~\ref{sec:simu}. In Section~\ref{sec:conclusion}, it comes to the conclusions.

\section{Preliminary Groundwork}\label{sec:prel}
This section provides preliminary development on trajectory sensitivity, integral inequality, and other necessary background.

The trajectory sensitivity of a system of general ordinary differential equations (ODEs) is as follows.
\begin{Lm}\cite{ODE}\label{lm:ODE_sensitivity}
Consider a system of ordinary differential equations
\begin{align}\label{m:ODE}
\begin{split}
\frac{\mathrm{d}x}{\mathrm{d}t}&=f(t,x,\lambda)\\
x(t_0)&=x_0.
\end{split}
\end{align}
Suppose $f$ is continuous and has continuous partial derivatives with respect to $x$ and $\lambda$ in the region
$R_1=\{(t,x,\lambda):|t-t_0|\leq a,\|x-x_0\|\leq b,\|\lambda-\lambda_0\|\leq c\}$.
Then the solution of~\eqref{m:ODE} is continuously differentiable with respect to $t$ in the region $R_2=\{(t,\lambda):|t-t_0|\leq h,\|\lambda-\lambda_0\|\leq c\}$, where $h=\min(a,\frac{b}{M})$ and $M=\max_{(t,x,\lambda)\in R_1}{|f|}.$
Let $x_s$ be the solution of~\eqref{m:ODE}.
Let $A=\frac{\partial f(t,x_s,\lambda)}{\partial x}$ and $B=\frac{\partial f(t,x_s,\lambda)}{\partial\lambda}$ be the partial derivatives of $f$ with respect to $x$ and $\lambda$, respectively, at $(t,x_s,\lambda)$.
Let $z_1=\frac{\partial x_s}{\partial t_0}$, $z_2=\frac{\partial x_s}{\partial x_0}$, and $z_3=\frac{\partial x_s}{\partial\lambda}$ be the derivatives, which are also called sensitivities, of the solution $x_s$ with respect to $t_{0}$, $x_{0}$, and $\lambda$, respectively.
Then the sensitivities obey the following ODEs:
\begin{align}
\frac{\mathrm{d}z_1}{\mathrm{d}t}&=Az_1,& z_1(t_0)&=-f(t_0,x_0,\lambda).\label{m:sensitivity_t0}\\
\frac{\mathrm{d}z_2}{\mathrm{d}t}&=Az_2, & z_2(t_0)&=1.\label{m:sensitivity_x0}\\
\frac{\mathrm{d}z_3}{\mathrm{d}t}&=Az_3+B,& z_3(t_0)&=0.\label{m:sensitivity_lbd}
\end{align}
\end{Lm}
\begin{R}
In Lemma~\ref{lm:ODE_sensitivity}, the variable $t$ is limited to a finite range $|t-t_0|\leq h.$ However, if $f$ is uniformly bounded with respect to $t$, we can extend $R_2$ in Lemma~\ref{lm:ODE_sensitivity} to a larger region with respect to $t$. Specifically, if $f$  depends only on $x$ and $\lambda$, we can extend $t$ to infinity in the region $R_2$.
\end{R}
\begin{R}
Note that the state matrix $A=\frac{\partial f(t,x_s,\lambda)}{\partial x}$ in~\eqref{m:sensitivity_t0}, \eqref{m:sensitivity_x0}, and \eqref{m:sensitivity_lbd} is time-varying and independent of the sensitivities $z_{1}$, $z_{2}$, and $z_{3}$. Therefore, \eqref{m:sensitivity_t0}, \eqref{m:sensitivity_x0}, and \eqref{m:sensitivity_lbd} are LTV systems.
\end{R}

For a scalar integral inequality, Gr{\"o}nwall's inequality applies, as follows.
\begin{Lm}[Gr{\"o}nwall's Lemma]\cite{Gronwall}\label{lm:gronwall}
Let $\alpha(t),\beta(t),\gamma(t)$ be a real continuous function defined in $[a,b]$.
If $$\gamma(t)\leq \alpha(t)+\int_{a}^{t}\beta(s)\gamma(s)\mathrm{d}s, \forall t\in [a,b]$$ and $\beta(t)\geq 0, \forall t\in [a,b]$, then
\begin{align}\label{m:gronwall}
\gamma(t)\leq \alpha(t)+\int_{a}^{t}\alpha(s)\beta(s)e^{\int_{s}^{t}\beta(\rho)\mathrm{d}\rho}\mathrm{d}s,\forall t\in [a,b].
\end{align}
If, in addition, the function $\alpha$ is non-decreasing, then
\begin{align*}
\gamma(t)\leq \alpha(t)e^{\int_{a}^{t}\beta(s)\mathrm{d}s},\forall t\in [a,b].
\end{align*}
\end{Lm}

An ODE can be equivalently converted to an integral equation, and then transformed into an integral inequality. 
Thus, Gr{\"o}nwall's Lemma can be used to compute bounds of the states of the trajectory sensitivity equations ~\eqref{m:sensitivity_t0}, \eqref{m:sensitivity_x0}, and \eqref{m:sensitivity_lbd}.

Matrix exponentials are important for linear systems.
We have the following lemma on the norm of matrix exponentials.
\begin{Lm}\cite{Hu2004MatExp}\label{lm:mat_exp}
Let $A$ be a stable matrix, (i.e., the real parts of its eigenvalues
are negative), and $H$ be the symmetric and positive-definite matrix such that $A^{T}H+HA^{T}=-I$, $\lambda_{max}(H)$ represents the maximum eigenvalue of $H$ and $\lambda_{min}(H)$ the minimum one.
Then, $\|e^{At}\|\leq \beta e^{-ct}$, where $\beta=\sqrt{\frac{\lambda_{max}(H)}{\lambda_{min}(H)}}$ and $c=\frac{1}{\lambda_{max}(H)}$.
\end{Lm}

From Lemma~\ref{lm:mat_exp}, the norm of the exponential of a stable matrix can be upper-bounded by a vanishing exponential function. Such a result is useful to investigate the steady state of a linear system.

The following lemma may be helpful for investigation of a linear system with vanishing input.
\begin{Lm}\label{lm:int_0}
\begin{align}
\lim\limits_{t\to\infty}\int_{0}^{t}e^{-c(t-\tau)}\gamma(\tau)\mathrm{d}\tau=0,
\end{align}
where $c$ is a positive constant, $\gamma(t)\geq 0$ and $\lim_{t\to\infty}\gamma(t)=0$.
\end{Lm}
The proof of Lemma~\ref{lm:int_0} may be found in a calculus textbook. We re-prove it in the appendix of this paper.

\section{Bounds of Linear Time-Varying Systems}\label{sec:LTV_bound}
Model gaps are caused by imprecise parameters and unmodeled dynamics.
Bounds of model gaps resulting from these two causes can be captured by analyzing the bound of trajectory sensitivity of the model. 
Since the trajectory sensitivity of a model is described by LTV systems, bounds of LTV systems are derived in this section.


Consider an LTV system
\begin{align}\label{m:LTV}
\begin{split}
\frac{\mathrm{d}z}{\mathrm{d}t}&=A(t)z+u(t)\\
z(0)&=0.
\end{split}
\end{align}
The following assumption is imposed on the state matrix of the LTV system~\eqref{m:LTV}:
\begin{Asmp}\label{asmp:A}
$A(\infty)=\lim_{t\to\infty}A(t)$ exists and is a constant matrix, and the eigenvalues of $A(\infty)$ have negative real parts.
\end{Asmp}
\begin{R}
Assumption~\ref{asmp:A} is reasonable because practical power plants usually reach a steady state, which results in constant-limit matrices in its trajectory sensitivity.

\end{R}

In this section, we consider two cases:
\begin{inparaenum}
\item the input to an LTV system is known, and
\item the input to an LTV system is not known but bounded.
\end{inparaenum}
 
The first case can be used to quantify model gaps caused by imprecise parameters and the second one can be applied to obtaining model gaps arising from unmodeled dynamics.

\subsection{Bounds of LTV Systems with Known Input}
When the input of the LTV system~\eqref{m:LTV} of which the state matrix has a constant limit as stated in Assumption~\ref{asmp:A} can be accurately known,
bounds of the state of the LTV system can be obtained by decomposing the LTV system into a linear time-invariant (LTI) system of which the state matrix is $A(\infty)$ as in Assumption~\ref{asmp:A} and an additional term. The result is stated as follows.
\begin{T}\label{thm:LTV_bound}
Consider the LTV system in~\eqref{m:LTV} which satisfies Assumption~\ref{asmp:A}.
Let
\begin{align}\label{m:z_LTI}
z_{LTI}(t)=e^{A(\infty)t}\int_{0}^{t}e^{-A(\infty)\tau}u(\tau)\mathrm{d}\tau,
\end{align}
let $g(t)$ be an upper bound of \[\int_{0}^{t}\|e^{A(\infty)(t-\tau)}(A(\tau)-A(\infty))z_{LTI}(\tau)\|\mathrm{d}\tau,\] and let $h(t,\tau)$ be an upper bound of \[\|e^{A(\infty)(t-\tau)}(A(\tau)-A(\infty))\|.\]
Then, the bounds of the state of the LTV system~\eqref{m:LTV} can be estimated in the following two ways:
\begin{asparaenum}[Bound 1)]
\item \label{T1:B1} The distance between $z$ and $z_{LTI}$ can be bounded as follows:
\begin{align}\label{m:z_norm_diff}
\begin{split}
\|z(t)-z_{LTI}(t)\|
\leq  g(t)+\max_{0\leq \tau\leq t}\{g(\tau)\}\left(e^{\int_{0}^{t}h(t,s)\mathrm{d}s}-1\right).
\end{split}
\end{align}
Thus, the lower bound of the $i$\textsuperscript{th} entry of $z$ is \[z_{LTI,i}(t)-g(t)-\max_{0\leq \tau\leq t}\{g(\tau)\}\left(e^{\int_{0}^{t}h(t,s)\mathrm{d}s}-1\right)\]
and the upper bound is
\[z_{LTI,i}(t)+g(t)+\max_{0\leq \tau\leq t}\{g(\tau)\}\left(e^{\int_{0}^{t}h(t,s)\mathrm{d}s}-1\right),\]
where $z_{LTI,i}$ is the $i$\textsuperscript{th} entry of $z_{LTI}$.
\item \label{T1:B2} The norm of $z$ can be bounded as follows:
\begin{align}
\|z(t)\|\leq \|z_{LTI}(t)\|+\max_{0\leq \tau\leq t}\{\|z_{LTI}(\tau)\|\}(e^{\int_{0}^{t}h(t,s)\mathrm{d}s}-1 ).
\end{align}
Thus, the lower bound of the $i$\textsuperscript{th} entry of $z$ is
\[-\|z_{LTI}(t)\|-\max_{0\leq \tau\leq t}\{\|z_{LTI}(\tau)\|\}(e^{\int_{0}^{t}h(t,s)\mathrm{d}s}-1 )\]
and the upper bound is
\[\|z_{LTI}(t)\|+\max_{0\leq \tau\leq t}\{\|z_{LTI}(\tau)\|\}(e^{\int_{0}^{t}h(t,s)\mathrm{d}s}-1 ).\]
\end{asparaenum}
\end{T}
\begin{proof}
It can be obtained from~\eqref{m:LTV} that
\begin{align*}
\frac{\mathrm{d}z}{\mathrm{d}t}&=A(t)z+u(t) \\
&=A(\infty)z+u(t)+(A(t)-A(\infty))z.
\end{align*}
Let $\beta(t)=u(t)+(A(t)-A(\infty))z.$ Then
\begin{align*}
\frac{\mathrm{d}z}{\mathrm{d}t}=A(\infty)z+\beta(t).
\end{align*}
Then
\begin{align}\label{m:z_solution}
\begin{split}
z(t)=&e^{A(\infty)t}\int_{0}^{t}e^{-A(\infty)\tau}\beta(\tau)\mathrm{d}\tau\\
=&e^{A(\infty)t}\int_{0}^{t}e^{-A(\infty)\tau}u(\tau)\mathrm{d}\tau\\
&+e^{A(\infty)t}\int_{0}^{t}e^{-A(\infty)\tau}(A(\tau)-A(\infty))z(\tau)\mathrm{d}\tau\\
=&z_{LTI}(t)+e^{A(\infty)t}\int_{0}^{t}e^{-A(\infty)\tau}(A(\tau)-A(\infty))z(\tau)\mathrm{d}\tau.
\end{split}
\end{align}
\begin{asparaenum}[Proof of Bound 1):]
\item We first derive the upper bound of the distance between $z$ and $z_{LTI}$.
    From~\eqref{m:z_solution}, it can be deduced that
\begin{align}\label{m:z_LTI_dif}
\begin{split}
&z(t)-z_{LTI}(t)\\
=&e^{A(\infty)t}\int_{0}^{t}e^{-A(\infty)\tau}(A(\tau)-A(\infty))z(\tau)\mathrm{d}\tau\\
=&e^{A(\infty)t}\int_{0}^{t}e^{-A(\infty)\tau}(A(\tau)-A(\infty))(z(\tau)-z_{LTI}(\tau))\mathrm{d}\tau\\
&+e^{A(\infty)t}\int_{0}^{t}e^{-A(\infty)\tau}(A(\tau)-A(\infty))z_{LTI}(\tau)\mathrm{d}\tau
\end{split}
\end{align}
Then
\begin{align*}
&\|z(t)-z_{LTI}(t)\|\\
\leq& \int_{0}^{t}\|e^{A(\infty)(t-\tau)}(A(\tau)-A(\infty))z_{LTI}(\tau)\|\mathrm{d}\tau\\
&+\int_{0}^{t}\|e^{A(\infty)(t-\tau)}(A(\tau)-A(\infty))\|\|(z(\tau)-z_{LTI}(\tau))\|\mathrm{d}\tau\\
\leq & g(t)+\int_{0}^{t}h(t,\tau)\|(z(\tau)-z_{LTI}(\tau))\|\mathrm{d}\tau.
\end{align*}
Applying Lemma~\ref{lm:gronwall}, we find that
\begin{align}\label{m:z_gronwall}
\begin{split}
&\|z-z_{LTI}\|\\
\leq & g(t)+\int_{0}^{t}g(\tau)h(t,\tau)e^{\int_{\tau}^{t}h(t,s)\mathrm{d}s}\mathrm{d}\tau\\
\leq &g(t)+\max_{0\leq \tau\leq t}\{g(\tau)\}\int_{0}^{t}h(t,\tau)e^{\int_{\tau}^{t}h(t,s)\mathrm{d}s}\mathrm{d}\tau\\
=&g(t)+\max_{0\leq \tau\leq t}\{g(\tau)\}e^{\int_{0}^{t}h(t,s)\mathrm{d}s}\int_{0}^{t}h(t,\tau)e^{-\int_{0}^{\tau}h(t,s)\mathrm{d}s}\mathrm{d}\tau\\
\end{split}
\end{align}
In~\eqref{m:z_gronwall},
\begin{align*}
&\int_{0}^{t} h(t,\tau)e^{-\int_{0}^{\tau}h(t,s)\mathrm{d}s}\mathrm{d}\tau\\
=&\int_{0}^{t}  e^{-\int_{0}^{\tau}h(t,s)\mathrm{d}s}\mathrm{d}(\int_{0}^{\tau}h(t,s)\mathrm{d}s)\\
=&-\left(e^{-\int_{0}^{\tau}h(t,s)\mathrm{d}s} \right)\Big|_{0}^{t}\\
=&1 -e^{-\int_{0}^{t}h(t,s)\mathrm{d}s}
\end{align*}
Therefore,
\begin{align*}
\|z(t)-z_{LTI}(t)\|
\leq  g(t)+\max_{0\leq \tau\leq t}\{g(\tau)\}(e^{\int_{0}^{t}h(t,s)\mathrm{d}s}-1).
\end{align*}
Let $z_{i}$ and $z_{LTI,i}$ denote the $i$\textsuperscript{th} entries of $z$ and $z_{LTI}$, respectively. It can be deduced that $|z_{i}-z_{LTI,i}|\leq \|z-z_{LTI}\|$ and thus
\[z_{LTI,i}-\|z-z_{LTI}\|\leq z_{i}\leq z_{LTI,i}+ \|z-z_{LTI}\|.\]
Therefore, the lower bound of $z_{i}$ is
\[z_{LTI,i}(t)-g(t)-\max_{0\leq \tau\leq t}\{g(\tau)\}\left(e^{\int_{0}^{t}h(t,s)\mathrm{d}s}-1\right)\]
and the upper bound is
\[z_{LTI,i}(t)+g(t)+\max_{0\leq \tau\leq t}\{g(\tau)\}\left(e^{\int_{0}^{t}h(t,s)\mathrm{d}s}-1\right).\]

\item We now derive a bound of $\|z\|$. Such a bound can be easily obtained from the derivation of Bound \ref{T1:B1}). However, in this part, we derive a smaller bound of $z$ with fewer steps than that derived from Bound \ref{T1:B1}).

    From~\eqref{m:z_solution}, it can be deduced that
\begin{align*}
&\|z(t)\|\\
\leq & \|z_{LTI}(t)\|+\int_{0}^{t}\|e^{A(\infty)(t-\tau)}(A(\tau)-A(\infty))\|\|z(\tau)\|\mathrm{d}\tau\\
\leq &\|z_{LTI}(t)\|+\int_{0}^{t}h(t,\tau)\|z(\tau)\|\mathrm{d}\tau.
\end{align*}
From Lemma~\ref{lm:gronwall},
\begin{align*}
&\|z(t)\|\\
\leq &
\|z_{LTI}(t)\|+\int_{0}^{t}\|z_{LTI}(\tau)\|h(t,\tau)e^{\int_{\tau}^{t}h(t,s)\mathrm{d}s}\mathrm{d}\tau\\
=&\|z_{LTI}(t)\|+e^{\int_{0}^{t}h(t,s)\mathrm{d}s}\int_{0}^{t}\|z_{LTI}(\tau)\|h(t,\tau)e^{-\int_{0}^{\tau}h(t,s)\mathrm{d}s}\mathrm{d}\tau\\
\leq & \|z_{LTI}(t)\|+e^{\int_{0}^{t}h(t,s)\mathrm{d}s}\max_{0\leq \tau\leq t}\{\|z_{LTI}(\tau)(t)\|\}\times\\
&\int_{0}^{t}h(t,\tau)e^{-\int_{0}^{\tau}h(t,s)\mathrm{d}s}\mathrm{d}\tau\\
=&\|z_{LTI}(t)\|+\max_{0\leq \tau\leq t}\{\|z_{LTI}(\tau)\|\}\left(e^{\int_{0}^{t}h(t,s)\mathrm{d}s}-1\right).
\end{align*}
Let $z_{i}$ denote the $i$\textsuperscript{th} entry of $z$.
From the fact that $-\|z\|\leq z_{i}\leq \|z\|$, it can be determined that
the lower bound of $z_{i}$ is
\[-\|z_{LTI}(t)\|-\max_{0\leq \tau\leq t}\{\|z_{LTI}(\tau)\|\}(e^{\int_{0}^{t}h(t,s)\mathrm{d}s}-1 )\]
and the upper bound is
\[\|z_{LTI}(t)\|+\max_{0\leq \tau\leq t}\{\|z_{LTI}(\tau)\|\}(e^{\int_{0}^{t}h(t,s)\mathrm{d}s}-1 ).\]
\end{asparaenum}
\end{proof}
\begin{R}
A tighter bound can be obtained for Bound~\ref{T1:B1} via a more accurate derivation of \eqref{m:z_gronwall} as follows.
\begin{align*}
&\|z-z_{LTI}\|\\
\leq & g(t)+\int_{0}^{t}g(\tau)h(t,\tau)e^{\int_{\tau}^{t}h(t,s)\mathrm{d}s}\mathrm{d}\tau\\
\leq & g(t)+e^{\int_{0}^{t}h(t,s)\mathrm{d}s}\sum\limits_{i=0}^{N-1}\max_{t_{i}\leq \tau\leq t_{i+1}}\{g(\tau)\}\times\\
&\int_{t_{i}}^{t_{i+1}}h(t,\tau)e^{-\int_{0}^{\tau}h(t,s)\mathrm{d}s}\mathrm{d}\tau,
\end{align*}
where $t_{0}=1$, $t_{N}=t$, and $t_{i}<t_{i+1},i=0,1,\cdots,N-1$.
Accordingly,
\begin{align}\label{m:bound1_tight}
\begin{split}
&\|z(t)-z_{LTI}(t)\|\\
\leq & g(t)+\sum\limits_{i=0}^{N-1}\max_{t_{i}\leq \tau\leq t_{i+1}}\{g(\tau)\}e^{\int_{0}^{t}h(t,s)\mathrm{d}s}\times\\
&\qquad \quad (e^{-\int_{0}^{t_{i}}h(t,s)\mathrm{d}s}-e^{-\int_{0}^{t_{i+1}}h(t,s)\mathrm{d}s}).
\end{split}
\end{align}
The bound in~\eqref{m:bound1_tight} is tighter than the bound in Bound \ref{T1:B1}, but finding it requires more computation.
\end{R}

With the derived bounds, the behavior of $z(t)$ can be approximated by that of $z_{LTI}(t)$ when $t$ is sufficiently large, as shown below.
\begin{C}\label{coro:lim}
Consider an LTV system in~\eqref{m:LTV} with Assumption~\ref{asmp:A}. Also, assume that $u(t)$ is bounded. Let $z_{LTI}$ be defined as in~\eqref{m:z_LTI}. Then
\begin{align*}
\lim\limits_{t\to\infty}(z(t)-z_{LTI}(t))=0.
\end{align*}
\end{C}
\begin{proof}
Since $u(t)$ is bounded and $A(\infty)$ is stable, $z_{LTI}$ is also bounded. Let $M_{LTI}$ be a bound of $\|z_{LTI}\|$.
$g(t)$ in Theorem~\ref{thm:LTV_bound} can be selected as $$\int_{0}^{t}\|e^{A(\infty)(t-\tau)}\|\|(A(\tau)-A(\infty))\|\|z_{LTI}(\tau)\|\mathrm{d}\tau$$
and $h(t,\tau)$ as
$$\|e^{A(\infty)(t-\tau)}\|\|(A(\tau)-A(\infty))\|.$$
Then,
\[g(t)=\int_{0}^{t}h(t,\tau)\|z_{LTI}(\tau)\|\mathrm{d}\tau\leq M_{LTI}\int_{0}^{t}h(t,\tau)\mathrm{d}\tau.\]
From Lemma~\ref{lm:mat_exp}, $\|e^{A(\infty)(t-\tau)}\|\leq \beta e^{-c(t-\tau)}$, where $\beta$ and $c$ are selected as in Lemma~\ref{lm:mat_exp}.
Thus, \[\int_{0}^{t}h(t,\tau)\mathrm{d}\tau\leq \beta \int_{0}^{t} e^{-c(t-\tau)}\|(A(\tau)-A(\infty))\|\mathrm{d}\tau. \]
Note that $\lim_{\tau\to\infty}\|(A(\tau)-A(\infty))\|=0$.
Then, from Lemma~\ref{lm:int_0},
\[\lim\limits_{t\to\infty}\int_{0}^{t} e^{-c(t-\tau)}\|(A(\tau)-A(\infty))\|\mathrm{d}\tau=0.\]
It can be determined that
$$\lim\limits_{t\to\infty}\int_{0}^{t} h(t,\tau)\mathrm{d}\tau\leq 0.$$
Also, as $h(t,\tau)\geq 0$ and thus $\int_{0}^{t} h(t,\tau)\mathrm{d}\tau\geq 0$,
\[\lim\limits_{t\to\infty}\int_{0}^{t} h(t,\tau)\mathrm{d}\tau=0.\]
So, $\lim\limits_{t\to\infty}g(t)\leq 0$.
As $g(t)\geq 0$, $\lim\limits_{t\to\infty}g(t)=0.$
Also, from the fact that
\[\lim\limits_{t\to\infty}\int_{0}^{t} h(t,\tau)\mathrm{d}\tau=0\]
it can be shown that
$\lim\limits_{t\to\infty}e^{\int_{0}^{t}h(t,s)\mathrm{d}s}-1=0$.
From Bound~\ref{T1:B1}), $\lim\limits_{t\to\infty}(z(t)-z_{LTI}(t))=0.$
\end{proof}

\subsection{Bounds of LTV Systems with Unknown but Bounded Input}
When the disturbance to the ODEs in~\eqref{m:ODE} is constant, the bounds derived in Theorem~\ref{thm:LTV_bound} can be directly applied to quantifying the bound of the trajectory sensitivity of~\eqref{m:ODE}. In this subsection, we thus focus on the case in which the disturbance varies with time.

Let $\delta\lambda(t)$ be an arbitrary continuous function such that $\|\delta\lambda(t)\|_{\infty}=1$ and let $\eta$ be a small positive number.
$\eta\delta\lambda(t)$ can represent an arbitrary continuous signal of which the magnitude is $\eta$.
The time-varying disturbance of the parameter $\lambda$ in~\eqref{m:ODE} can be expressed as $\eta\delta\lambda(t)$. Let $\lambda^{(\eta)}(t)=\lambda+\eta\delta\lambda(t)$.
Then when $\eta=0$, $\lambda^{(0)}(t)=\lambda$ and there is no disturbance.
Let $x^{(\eta)}$ be the solution of the following ODE:
\begin{align*}
\frac{\mathrm{d}x}{\mathrm{d}t}&=f(t,x,\lambda^{(\eta)})\\
x(t_0)&=x_0.
\end{align*}
Then $z^{(\eta)}=\frac{\mathrm{d}x^{(\eta)}(t)}{\mathrm{d}\eta}$ is the sensitivity of $x^{(\eta)}$ with respect to $\eta$.
Then from Lemma~\ref{lm:ODE_sensitivity},
\begin{align}\label{m:sens_TV}
\frac{\mathrm{d}z^{(0)}}{\mathrm{d}t}=\frac{\partial f}{\partial x}z^{(0)}+\frac{\partial f}{\partial \lambda}\delta\lambda
\end{align}
and has initial values $z^{(0)}(t_{0})=0$.

If we know the mathematical form of $\delta\lambda$, then we can apply the bounds in Theorem~\ref{thm:LTV_bound} to obtain the bounds of the trajectory sensitivity of~\eqref{m:ODE} when the disturbance varies with time. However, the mathematical form of $\delta\lambda$ is unknown or cannot be accurately known in many cases. Therefore, a bound of LTV systems for bounded disturbances 
is necessary in those cases. Such a bound is derived as follows.
\begin{T}\label{thm:sens_all}
Consider an LTV system~\eqref{m:LTV} with Assumption~\ref{asmp:A}. Also, assume that $u(t)$ is a vector of which every entry is a continuous function.
Let $M_{u}=\sup_{t\in [0,\infty]}\max_{i}|u_{i}(t)|$, where $u_{i}$ is the $i$\textsuperscript{th} entry of $u$.
Let $z_{LTI}$, $g(t)$, and $h(t,\tau)$ be defined as in Theorem~\ref{thm:LTV_bound} and $a(t)$ be an upper bound of $\int_{0}^{t}\|e^{A(\infty)(t-\tau)}\|\mathrm{d}\tau$.
Then, the bound of the state of the LTV system can be estimated as follows:
\begin{align*}
\|z\|\leq K_{2,\infty}M_{u}\left(a(t)+\max_{0\leq \tau\leq t}\{a(t)\}(e^{\int_{0}^{t}h(t,s)\mathrm{d}s}-1)\right),
\end{align*}
where $K_{2,\infty}$ is a constant such that $\|u(t)\|_{2}\leq K_{2,\infty}\|u(t)\|_{\infty},\forall t$.
(It is easy to see that $K_{2,\infty}\leq \sqrt{n}$, where $n$ is the dimension of $u$.
)
Thus, the lower bound of every entry of $z$ is
\[-K_{2,\infty}M_{u}\left(a(t)+\max_{0\leq \tau\leq t}\{a(t)\}(e^{\int_{0}^{t}h(t,s)\mathrm{d}s}-1)\right)\]
and the upper bound is
\[K_{2,\infty}M_{u}\left(a(t)+\max_{0\leq \tau\leq t}\{a(t)\}(e^{\int_{0}^{t}h(t,s)\mathrm{d}s}-1)\right).\]
\end{T}
\begin{proof}
To obtain a bound of $\|z\|$ with bounded input $u(t)$, we can first obtain a bound of $\|z_{LTI}\|$ and then combine the bound of $\|z_{LTI}\|$ with Bound~\ref{T1:B2}) in Theorem~\ref{thm:LTV_bound}. Thus, we next derive a bound of $\|z_{LTI}\|$. From the definition of $z_{LTI}$ in~\eqref{m:z_LTI},
\begin{align*}
\|z_{LTI}(t)\|&\leq\int_{0}^{t}\|e^{A(\infty)(t-\tau)}\|\|u(\tau)\|\mathrm{d}\tau\\
&\leq K_{2,\infty}\int_{0}^{t}\|e^{A(\infty)(t-\tau)}\|\|u(\tau)\|_{\infty}\mathrm{d}\tau\\
&\leq K_{2,\infty}M_{u}\int_{0}^{t}\|e^{A(\infty)(t-\tau)}\|\mathrm{d}\tau\\
&\leq K_{2,\infty}M_{u} a(t).
\end{align*}
Therefore,
\[\|z\|\leq K_{2,\infty}M_{u}\left(a(t)+\max_{0\leq \tau\leq t}\{a(t)\}(e^{\int_{0}^{t}h(t,s)\mathrm{d}s}-1 )\right).\]
The lower bound and upper bound of every entry of $z$ are
\[-K_{2,\infty}M_{u}\left(a(t)+\max_{0\leq \tau\leq t}\{a(t)\}(e^{\int_{0}^{t}h(t,s)\mathrm{d}s}-1)\right)\]
and
\[K_{2,\infty}M_{u}\left(a(t)+\max_{0\leq \tau\leq t}\{a(t)\}(e^{\int_{0}^{t}h(t,s)\mathrm{d}s}-1)\right),\]
respectively.
\end{proof}
\begin{R}
The bound in Theorem~\ref{thm:sens_all} is derived for all bounded inputs. It might be very large for some types of bounded inputs.
\end{R}

\section{Case Studies}\label{sec:simu}
In this section, the bounds obtained in Section~\ref{sec:LTV_bound} are applied to quantifying a generator model gap caused by mismatch of mechanical power via simulation. 
The bound in Theorem~\ref{thm:LTV_bound} is applied to quantifying the gap caused by imprecise parameters, while the bound in Theorem~\ref{thm:sens_all} is applied to quantifying the gap caused by unmodeled dynamics.
Both causes are regarded as disturbances to the model.
For a generator, the gap of the rotor angle $\delta$ is more important than angular frequency $\omega$ because it is related to the gap of the electrical power output of the generator.
In this section, the generator response without any disturbance is referred to as nominal response, while that with disturbance as disturbed response.
Assume a generator is represented by the second-order generator model shown below:
\begin{align*}
\frac{\mathrm{d}\delta}{\mathrm{d}t}&=\Omega_b\left(\omega-1\right)\\
\frac{d\omega}{\mathrm{d}t}&=\frac{1}{M}\left(P_m-P_e-D\left(\omega-1\right)\right)\\
P_e&=\left(v_q+r_ai_q\right)i_q+\left(v_d+r_ai_d\right)i_d\\
Q_e&=\left(v_q+r_ai_q\right)i_d-\left(v_d+r_ai_d\right)i_q\\
0&=v_q+r_ai_q-e'_q+\left(x'_d-x_l\right)i_d\\
0&=v_d+r_ai_d-e'_d-\left(x'_q-x_l\right)i_q\\
v_d&=V\cos{\delta}\\
v_q&=V\sin{\delta},
\end{align*}
The physical meanings, values, and units of these parameters are collected in Table~\ref{tab:para_simu} and $\Omega_b=2\pi f$.
\begin{table}
\centering
\caption{Parameter Values in Simulation Example\label{tab:para_simu}}
\begin{tabular}{c|l|c|c}
Variable Name&Physical Meaning & Value& Unit\\
\hline
$M$&mechanical starting time&13& kWs/kVA\\
\hline
$f$&frequency rating&60&Hz\\ \hline
$x_d^\prime$&d-axis transient reactance&0.2&p.u.\\\hline
$x_q^\prime$&q-axis transient reactance&0.4&p.u.\\\hline
$x_l$&leakage reactance &0.15&p.u.\\\hline
$V$&voltage rating&1&p.u.\\\hline
$e_d^\prime$&d-axis transient potential  &0.1&p.u.\\\hline
$e_q^\prime$& q-axis transient potential &0.9&p.u.\\\hline
$r_a$&armature resistance&0.0005&p.u.\\\hline
$P_m$&mechanical power&1&p.u.\\\hline
$D$&damping coefficient&100&\hrulefill\\\hline
\end{tabular}
\end{table}

By solving $i_d, i_q$ as a function of $v_d,v_q$ and substituting the last two equations into it, we can express $P_e$ as a function of $\delta$. We denote the function $P_e=P_e(\delta)$. Let $P_{e,\delta}=\frac{\mathrm{d}P_{e}}{\mathrm{d}\delta}$. The generator model can be simplified to the following:
\begin{align}\label{m:gene_model_simp}
\begin{split}
\frac{\mathrm{d}\delta}{\mathrm{d}t}&=\Omega_b\left(\omega-1\right)\\
\frac{d\omega}{\mathrm{d}t}&=\frac{1}{M}\left(P_m-P_{e}(\delta)-D\left(\omega-1\right)\right).
\end{split}
\end{align}

\begin{figure}
\centering
\includegraphics[width=0.45\textwidth]{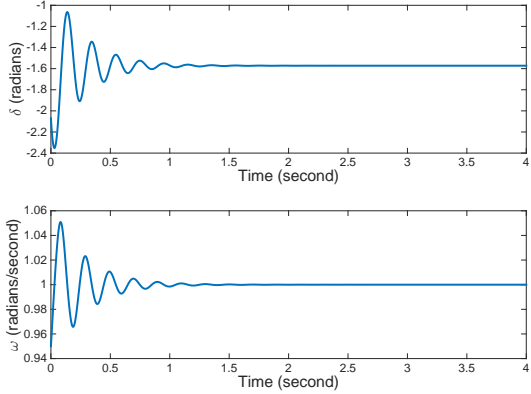}
\caption{Response of generator in nominal case}
\label{fig:nominal}
\end{figure}
The nominal response of the generator with parameter values being collected in Table~\ref{tab:para_simu} is presented in Fig.~\ref{fig:nominal}.
The initial values of $\delta$ and $\omega$ are selected to be $\delta_{0}=-\frac{\pi}{2}-0.5$ and $\omega_{0}=0.95$, respectively.
The generator model is stable with this pair of initial values.

Let $x=\left(\delta,\omega\right)^T$ represent the state of the generator and $f\left(x,P_{m}\right)=\left(\Omega_b\left(\omega-1\right),\left(P_m-P_e-D\left(\omega-1\right)\right)/M\right)^T$ be the right-hand side of the generator model.
Next, we only simulate the model gap caused by uncertainty of $P_m$. If we let $z=\frac{\mathrm{d} x}{\mathrm{d}\lambda}=\left(\begin{matrix}\frac{\mathrm{d}\delta}{\mathrm{d} P_m}&\frac{\mathrm{d}\omega}{\mathrm{d} P_m}\\\end{matrix}\right)^T$, then
\begin{align}\label{m:generator_sens}
\begin{split}
\frac{\mathrm{d}z}{\mathrm{d}t}&=A(t) z+f_{P_{m}}\left(x,P_{m}\right)\\
z\left(0\right)&=\left(0\ 0\right)^T,
\end{split}
\end{align}
where \[A(t)=f_x\left(x(t),P_{m}\right)=\left(\begin{matrix}0&\Omega_b\\\frac{-P_{e,\delta}}{M}&-\frac{D}{M}\\\end{matrix}\right),\]
and
\[f_{P_{m}}=\frac{\partial f}{\partial P_m}=\left(0,\frac{1}{M}\right)^T.\]
The trajectory sensitivity equations in~\eqref{m:generator_sens} are LTV, and its state matrix $A(t)$ is convergent as $\delta$ is convergent, which can be observed in Fig.~\ref{fig:nominal}.
When $\delta(\infty)$ is substituted into $A(t)$, $A(\infty)$ is stable and the imaginary parts of its eigenvalues are nonzero.
As a result, instead of the upper bound of matrix exponentials from Lemma~\ref{lm:mat_exp}, the upper bound of $\|e^{A(\infty)(t-\tau)}\|$ is selected to be $C_{A}e^{-k_{A}(t-\tau)}(1+S_{A}\sin(\omega_{A}(t-\tau)+\phi_{A}))$.
Note that the bound obtained from Lemma~\ref{lm:mat_exp} will result in bounds of a model gap that is too large to make sense.
The upper bound of $\|A(\tau)-A(\infty)\|$ is selected to be $C_{dA}e^{-k_{dA}(\tau)}(1+S_{dA}\sin(\omega_{dA}\tau+\phi_{dA}))$.
The way $A(\tau)-A(\infty)$ vanishes is determined by the way $\delta$ converges to zero in~\eqref{m:gene_model_simp}.
The generator model represented by~\eqref{m:gene_model_simp} is nonlinear. It is thus difficult to tell whether $\delta(\tau)$ and $A(\tau)$ converge in a way similar to $C_{dA}e^{-k_{dA}(\tau)}(1+S_{dA}\sin(\omega_{dA}\tau+\phi_{dA}))$.
However, we can always find such an upper bound of $\|A(\tau)-A(\infty)\|$.

Next, the bounds of LTV systems derived in Section~\ref{sec:LTV_bound} are applied to~\eqref{m:generator_sens} to estimate the model gap caused by imprecise parameters and unmodeled dynamics.
It should be noted that the bounds of the trajectory sensitivity equations in~\eqref{m:generator_sens} can only be applied when the disturbance is small.

\subsection{Bounds with imprecise Parameters}
\begin{figure}
\centering
\includegraphics[width=0.45\textwidth]{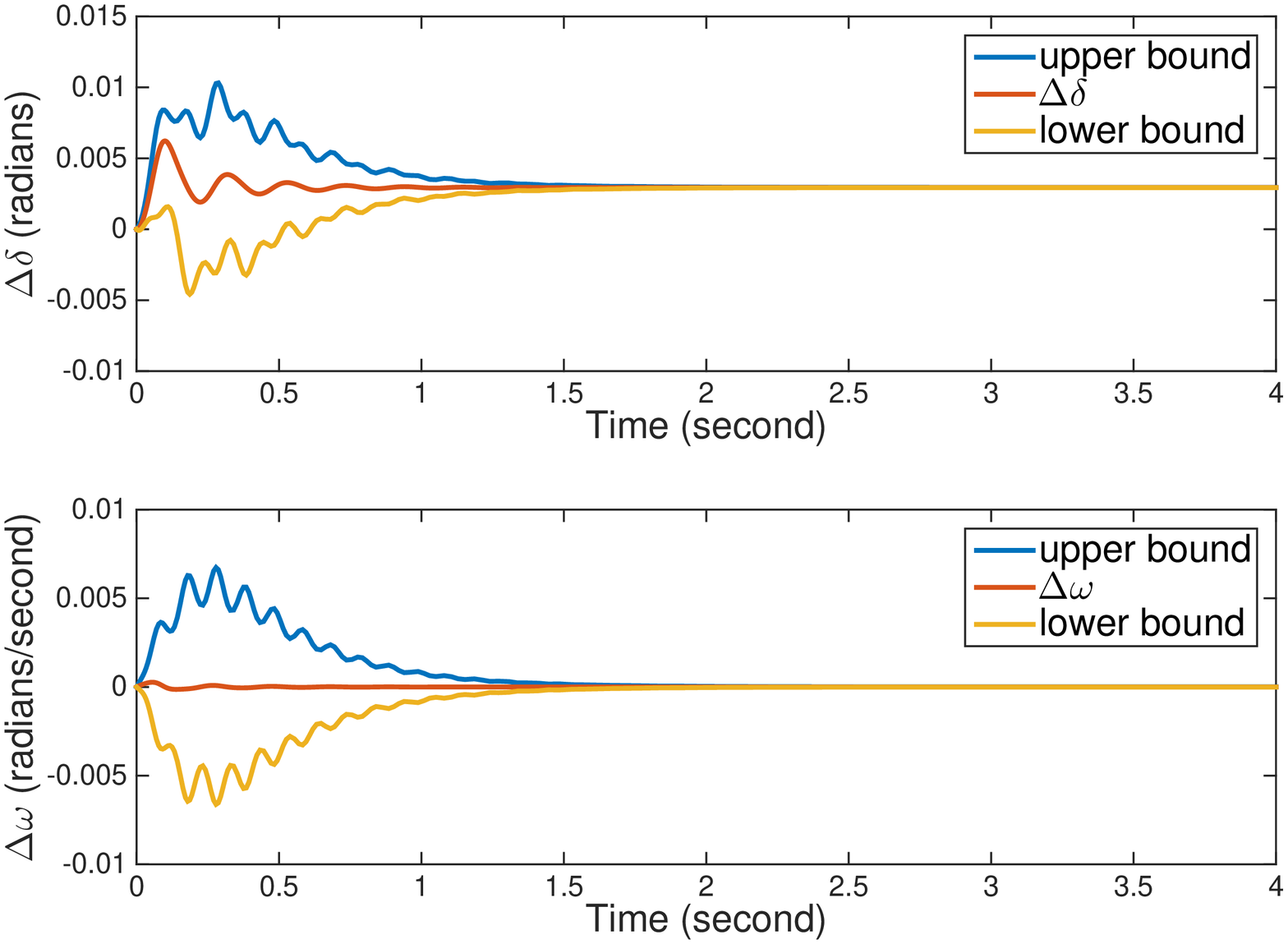}
\caption{Gap for generator with constant disturbance}
\label{fig:DisConst}
\end{figure}
The constant disturbance to $P_{m}$ is selected to be $0.1$p.u. 
The difference between the nominal response and the disturbed response is shown in Fig.~\ref{fig:DisConst}, which also shows the upper bound and the lower bound corresponding to those in Bound 1) in Theorem~\ref{thm:LTV_bound}. Note that the upper bound and the lower bound are expressed as $z_{LTI,i}(t) \pm \left(g(t)+\max_{0\leq \tau\leq t}\{g(\tau)\}(e^{\int_{0}^{t}h(t,s)\mathrm{d}s}-1)\right)$ and they are not symmetric about the x-axis.

The quantification of the model gap caused by constant disturbance is quite tight 
for the phase angle $\delta$. 
The obtained bounds have the same steady-state value as the actual model gap, and
the upper bound is very close to the actual model gap at the beginning.
In the middle of the response, when the time is around $0.5$ second, there is mismatch between the obtained bounds and the actual response gap.
That is because the upper bound estimates of $\|e^{A(\infty)(t-\tau)}\|$ and $\|A(\tau)-A(\infty)\|$ that are used in the simulation are larger than the actual $\|e^{A(\infty)(t-\tau)}\|$ and $\|A(\tau)-A(\infty)\|$.
The quantification of the model gap obtained here is quite accurate for the phase angle $\delta$.

Note that the bounds of the model gap are obtained for the state vector of the system, which is $\begin{pmatrix}\delta & \omega\end{pmatrix}^{T}$ in this example. If one state plays a major role in the gap, the obtained bounds may be tight for this state and very loose for other states.
This is why the quantification is not accurate for $\omega$, which can be observed in the plot at the bottom of Fig.~\ref{fig:DisConst}.
However, because the phase angle $\delta$ determines the power output of the generator and is more important than $\omega$, the quantification of model gaps works well enough for the generator model. 

\subsection{Bounds with Unmodeled Dynamics}
In this subsection, two time-varying disturbance signals are adopted, representing unmodeled dynamics.
The first one is a sine wave while the second one is the difference between the generator model in~\eqref{m:gene_model_simp} and a more complex one.

\begin{figure}
\centering
\includegraphics[width=0.45\textwidth]{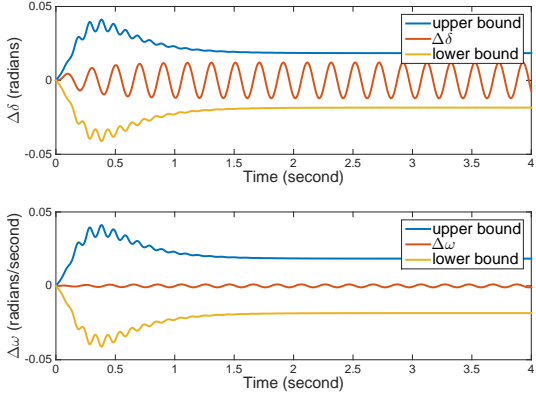}
\caption{Gap for generator with a sine wave disturbance}
\label{fig:sine}
\end{figure}
The magnitude of the sine wave is selected to be $0.1$p.u. 
and the frequency to be the resonant frequency of the system, $\dot{x}=A(\infty)x$.
The disturbance signal at the resonant frequency is amplified to the largest degree by the LTI 
system of which the state matrix is $A(\infty)$.
Combined with Corollary~\ref{coro:lim}, which states that the behavior of the LTV system of sensitivity equations in~\eqref{m:generator_sens} satisfying Assumption~\ref{asmp:A} can be approximated by the LTI system $\dot{z}=A(\infty)z+f_{P_{m}}$,
a sine signal at the generator's resonant frequency results in a larger model gap than other signals with the same magnitude.  
The model gap caused by this sine signal is presented in Figure~\ref{fig:sine}, in which the upper bound and the lower bound are those in Theorem~\ref{thm:sens_all} and are symmetric about the x-axis.
The bound is close for the disturbance of $\delta$, which shows that the bounds of model gaps are tight for this case.

\begin{figure}
\centering
\includegraphics[width=0.45\textwidth]{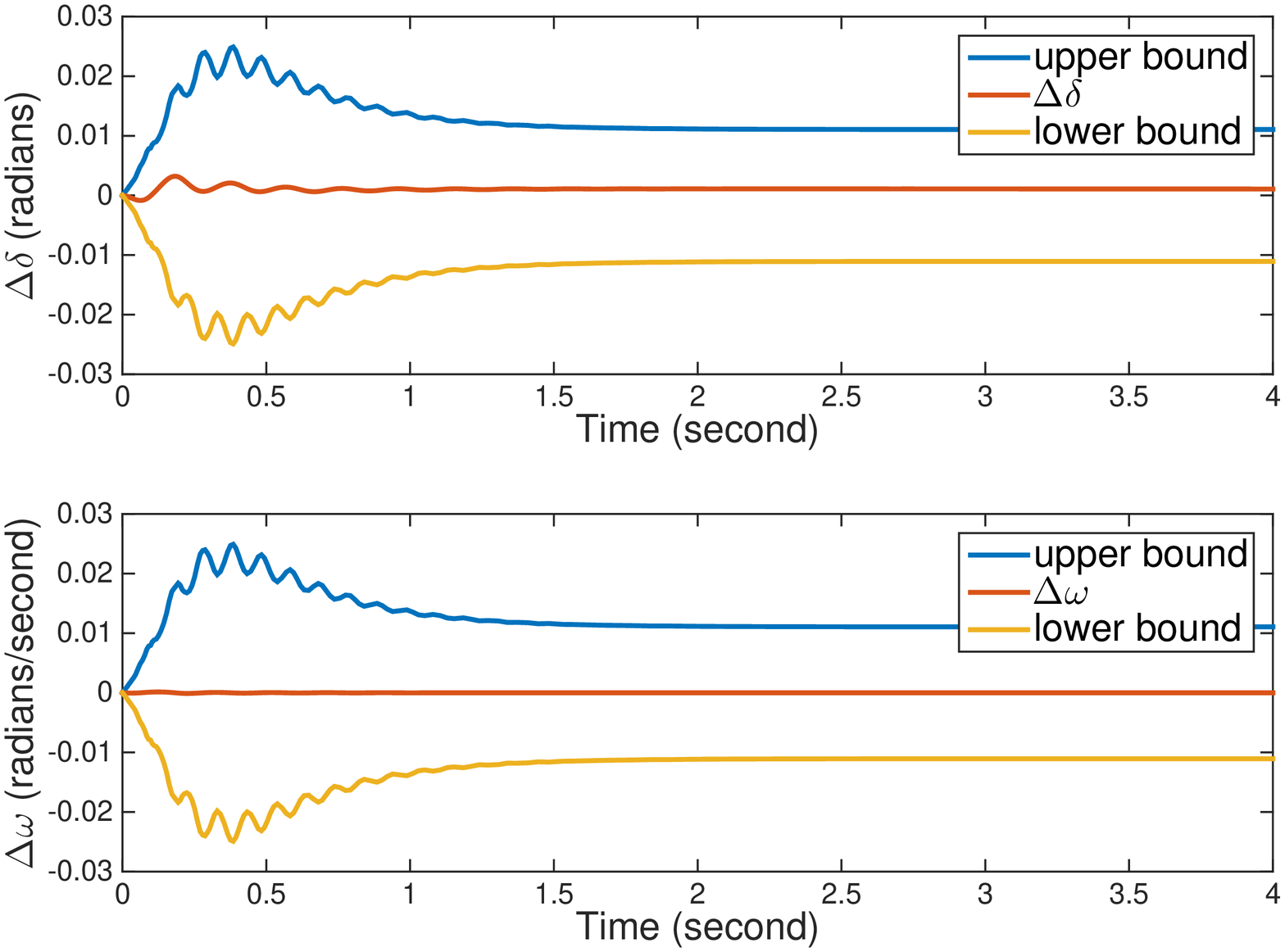}
\caption{Gap between second-order model and higher-order model}
\label{fig:disODE}
\end{figure}
The second signal for unmodeled dynamics originates from the difference between \eqref{m:gene_model_simp} and a higher-order model.
The higher-order model of the generator has the following definitions in addition to~\eqref{m:gene_model_simp},
\begin{align}\label{m:model_high}
\begin{split}
T_{in}^{\star}&=T_{\text{order}}+\frac{1}{R}(\omega_{\text{ref}}-\omega)\\
T_{in}&=\begin{cases}
T_{in}^{\star},&T_{\min}\leq T_{in}^{\star}\leq T_{\max},\\
T_{\max},&T_{in}^{\star}> T_{\max},\\
T_{\min},&T_{in}^{\star}<T_{\min},
\end{cases}\\
\dot{t}_{g1}&=\frac{T_{in}-t_{g1}}{T_{s}},\\
\dot{t}_{g2}&=\frac{(1-\frac{T_{3}}{T_{c}})t_{g1}-t_{g2}}{T_{c}},\\
\dot{t}_{g3}&=\frac{(1-\frac{T_{4}}{T_{5}})(t_{g2}+\frac{T_{3}}{T_{c}}t_{g1})-t_{g3}}{T_{5}},\\
P_{m}&=t_{g3}+\frac{T_{4}}{T_{5}}(t_{g2}+\frac{T_{3}}{T_{c}}t_{g1}),
\end{split}
\end{align}
where the physical meanings and values of the variables are presented in Table~\ref{tab:para_HO}.
\begin{table}
\centering
\caption{Parameter Values in Higher-order Model\label{tab:para_HO}}
\begin{tabular}{c|l|c|c}
\hline
Variable Name&Physical Meaning & Value& Unit\\
\hline
$\omega_{\text{ref}}$&reference speed &1 & p.u.\\
\hline
$R$&droop&0.02&p.u.\\
\hline
$T_{\max}$&maximum turbine output&1.2&p.u.\\
\hline
$T_{\min}$&minimum turbine output&0.3&p.u.\\
\hline
$T_{s}$&governor time constant&0.1& s\\ 
\hline
$T_{c}$&servo time constant&0.45&s\\
\hline
$T_{3}$&transient gain time constant&0&s\\
\hline
$T_{4}$&power fraction time constant&12&s\\
\hline
$T_{5}$&reheat time constant&50&s\\
\hline
\end{tabular}
\end{table}
Note that $P_{m}$ in \eqref{m:gene_model_simp} is regarded as a constant, while it is dynamic in \eqref{m:model_high}.
$P_{m}$ in~\eqref{m:model_high} is an output of an LTI system with saturated input. The gap between the higher-order model and the second-order model is presented in Fig.~\ref{fig:disODE},  in which the upper bound and the lower bound are those in Theorem~\ref{thm:sens_all} and are symmetric about the x-axis. It can be observed that the gap between the second-order model in~\eqref{m:gene_model_simp} and the higher-order model
is within the bound developed in Theorem~\ref{thm:sens_all}.

It should be noted that the bounds used in this subsection are developed for model gaps caused by all unmodeled dynamics. Thus, the bounds may be very loose for some specific disturbance, such as that resulting from a higher-order model, as presented in Fig.~\ref{fig:disODE}. But results for the sine signal selected for Fig.~\ref{fig:sine} show that the developed bounds are tight for the class of bounded dynamics.

\section{Conclusions}\label{sec:conclusion}
Bounds of model gaps were quantified by estimating the bounds of trajectory sensitivity.
First, bounds for responses of general linear time-varying systems were obtained with both constant input and dynamic input.
Then, the bounds for linear time-varying systems were applied to quantifying the model gaps. 
Simulations with a generator model showed the efficacy of the quantification.
Future work may include investigation of more general models and online quantification. 

\appendix[Proof of Lemma~\ref{lm:int_0}]
As $\gamma(t)\to 0$, $\|\gamma\|_{\infty}=\max_{0\leq t<\infty}\{\gamma(t)\}$ is finite.
And for any $\epsilon>0$, there exists $T_{\gamma}$, such that for all $\tau>T_{\gamma}$, $\gamma(\tau)<\frac{c}{\|\gamma\|_{\infty}+1}\epsilon$;
there also exists $T_{e}$, such that for all $t_{e}>T_{e}$, $e^{-ct}<\frac{c}{\|\gamma\|_{\infty}+1}\frac{\epsilon}{2}$.
Let $T=\max\{T_{\gamma},T_{e}\}$.
Then, for all $t>2T$,
\begin{align*}
&\int_{0}^{t}e^{-c(t-\tau)}\gamma(\tau)\mathrm{d}\tau\\
=&\int_{0}^{T}e^{-c(t-\tau)}\gamma(\tau)\mathrm{d}\tau+\int_{T}^{t}e^{-c(t-\tau)}\gamma(\tau)\mathrm{d}\tau\\
\leq & \max_{0\leq\tau\leq T}\{\gamma(\tau)\}\int_{0}^{T}e^{-c(t-\tau)}\mathrm{d}\tau+\frac{c}{\|\gamma\|_{\infty}+1}\epsilon\int_{T}^{t}e^{-c(t-\tau)}\mathrm{d}\tau\\
=&\max_{0\leq\tau\leq T}\{\gamma(\tau)\}\frac{e^{-c(t-T)-e^{-ct}}}{c}+\frac{c}{\|\gamma\|_{\infty}+1}\epsilon\frac{1-e^{-c(t-T)}}{c}\\
<&\max_{0\leq\tau\leq T}\{\gamma(\tau)\}\frac{e^{-c(t-T)+e^{-ct}}}{c}+\frac{c}{\|\gamma\|_{\infty}+1}\epsilon\frac{1-e^{-c(t-T)}}{c}\\
< &\frac{\|\gamma\|_{\infty}}{c}\frac{c}{\|\gamma\|_{\infty}+1}\epsilon+\frac{c}{\|\gamma\|_{\infty}+1}\frac{1}{c}\epsilon\\
=&\epsilon.
\end{align*}
As a result, $\lim\limits_{t\to\infty}\int_{0}^{t}e^{-c(t-\tau)}\gamma(\tau)\mathrm{d}\tau=0.$

\bibliographystyle{IEEEtran}
\bibliography{IEEEabrv,mybib}

\begin{thebibliography}{10}
\providecommand{\url}[1]{#1}
\csname url@rmstyle\endcsname
\providecommand{\newblock}{\relax}
\providecommand{\bibinfo}[2]{#2}
\providecommand\BIBentrySTDinterwordspacing{\spaceskip=0pt\relax}
\providecommand\BIBentryALTinterwordstretchfactor{4}
\providecommand\BIBentryALTinterwordspacing{\spaceskip=\fontdimen2\font plus
\BIBentryALTinterwordstretchfactor\fontdimen3\font minus
  \fontdimen4\font\relax}
\providecommand\BIBforeignlanguage[2]{{%
\expandafter\ifx\csname l@#1\endcsname\relax
\typeout{** WARNING: IEEEtran.bst: No hyphenation pattern has been}%
\typeout{** loaded for the language `#1'. Using the pattern for}%
\typeout{** the default language instead.}%
\else
\language=\csname l@#1\endcsname
\fi
#2}}

\bibitem{Kosterev2004hydro}
\BIBentryALTinterwordspacing
D.~Kosterev, ``Hydro turbine-governor model validation in {P}acific
  {N}orthwest,'' \emph{IEEE Transactions on Power Systems}, vol.~19, no.~2, pp.
  1144--1149, 2004. [Online]. Available:
  \url{https://ieeexplore.ieee.org/document/1295026}
\BIBentrySTDinterwordspacing

\bibitem{Huang2013PMU}
\BIBentryALTinterwordspacing
Z.~Huang, P.~Du, D.~Kosterev, and S.~Yang, ``Generator dynamic model validation
  and parameter calibration using phasor measurements at the point of
  connection,'' \emph{IEEE Transactions on Power Systems}, vol.~28, no.~2, pp.
  1939--1949, 2013. [Online]. Available:
  \url{https://ieeexplore.ieee.org/document/6487426}
\BIBentrySTDinterwordspacing

\bibitem{VendorTool}
\BIBentryALTinterwordspacing
{North American SynchroPhasor Initiative (NASPI)}, ``Technical workshop - model
  verification tools,'' 2016. [Online]. Available:
  \url{https://naspi.org/node/528}
\BIBentrySTDinterwordspacing

\bibitem{Pourbeik2013ModelValidation}
\BIBentryALTinterwordspacing
P.~Pourbeik, R.~Rhinier, S.~Hsu, B.~L. Agrawal, and R.~Bisbee, ``Semiautomated
  model validation of power plant equipment using online measurements,''
  \emph{IEEE Transactions on Energy Conversion}, vol.~28, no.~2, pp. 308--316,
  2013. [Online]. Available: \url{https://ieeexplore.ieee.org/document/6459572}
\BIBentrySTDinterwordspacing

\bibitem{EPRI}
\BIBentryALTinterwordspacing
{Electric Power Research Institute (EPRI)}, ``Power plant parameter derivation
  ({PPPD}),'' 2015. [Online]. Available:
  \url{https://www.epri.com/research/products/1022543}
\BIBentrySTDinterwordspacing

\bibitem{MathWorks}
\BIBentryALTinterwordspacing
{MathWorks}, ``Model-based calibration toolbox,'' 2017. [Online]. Available:
  \url{https://www.mathworks.com/products/mbc/}
\BIBentrySTDinterwordspacing

\bibitem{Nayak2016PMU}
\BIBentryALTinterwordspacing
N.~Nayak, H.~Chen, W.~Schmus, and R.~Quint, ``Generator parameter validation
  and calibration process based on {PMU} data,'' \emph{2016 IEEE/PES
  Transmission and Distribution Conference and Exposition (T\&D), Dallas, TX},
  pp. 1--5, 2016. [Online]. Available:
  \url{https://ieeexplore.ieee.org/abstract/document/7519886}
\BIBentrySTDinterwordspacing

\bibitem{Tsai2017PMU}
\BIBentryALTinterwordspacing
C.~Tsai, L.~Changchien, I.~Chen, C.~Lin, W.~Lee, C.~Wu, and H.~Lan, ``Practical
  considerations to calibrate generator model parameters using phasor
  measurements,'' \emph{IEEE Transactions on Smart Grid}, vol.~8, no.~5, pp.
  2228--2238, 2017. [Online]. Available:
  \url{https://ieeexplore.ieee.org/document/7401126}
\BIBentrySTDinterwordspacing

\bibitem{Hockenberry2004PCM}
\BIBentryALTinterwordspacing
J.~R. {Hockenberry} and B.~C. {Lesieutre}, ``Evaluation of uncertainty in
  dynamic simulations of power system models: The probabilistic collocation
  method,'' \emph{IEEE Transactions on Power Systems}, vol.~19, no.~3, pp.
  1483--1491, 2004. [Online]. Available:
  \url{https://ieeexplore.ieee.org/document/1318685}
\BIBentrySTDinterwordspacing

\bibitem{Borden2012NormalForm}
\BIBentryALTinterwordspacing
A.~R. {Borden} and B.~C. {Lesieutre}, ``Determining power system modal content
  of data motivated by normal forms,'' in \emph{2012 North American Power
  Symposium (NAPS), Champaign, IL}, 2012, pp. 1--6. [Online]. Available:
  \url{https://ieeexplore.ieee.org/document/6336387}
\BIBentrySTDinterwordspacing

\bibitem{Borden2013Tool}
\BIBentryALTinterwordspacing
A.~R. {Borden}, B.~C. {Lesieutre}, and J.~{Gronquist}, ``Power system modal
  analysis tool developed for industry use,'' in \emph{2013 North American
  Power Symposium (NAPS), Manhattan, KS}, 2013, pp. 1--6. [Online]. Available:
  \url{https://ieeexplore.ieee.org/document/6666956}
\BIBentrySTDinterwordspacing

\bibitem{Borden2014Proj}
\BIBentryALTinterwordspacing
A.~R. Borden and B.~C. Lesieutre, ``Variable projection method for power system
  modal identification,'' \emph{IEEE Transactions on Power Systems}, vol.~29,
  no.~6, pp. 2613--2620, 2014. [Online]. Available:
  \url{https://ieeexplore.ieee.org/document/1318685}
\BIBentrySTDinterwordspacing

\bibitem{Huang2013Identification}
\BIBentryALTinterwordspacing
R.~Huang, E.~Farantatos, G.~J. Cokkinides, and A.~P. Meliopoulos, ``Physical
  parameters identification of synchronous generators by a dynamic state
  estimator,'' \emph{2013 IEEE Power \& Energy Society General Meeting,
  Vancouver, BC}, pp. 1--5, 2013. [Online]. Available:
  \url{https://ieeexplore.ieee.org/document/6672882}
\BIBentrySTDinterwordspacing

\bibitem{Hajnoroozi2015PMU}
\BIBentryALTinterwordspacing
A.~A. Hajnoroozi, F.~Aminifar, and H.~Ayoubzadeh, ``Generating unit model
  validation and calibration through synchrophasor measurements,'' \emph{IEEE
  Transactions on Smart Grid}, vol.~6, no.~1, pp. 441--449, 2015. [Online].
  Available: \url{https://ieeexplore.ieee.org/document/696561}
\BIBentrySTDinterwordspacing

\bibitem{Huang2018Calibrating}
\BIBentryALTinterwordspacing
R.~Huang, R.~Diao, Y.~Li, J.~J. Sanchez-{G}asca, Z.~Huang, B.~Thomas, P.~V.
  Etingov, S.~Kincic, S.~Wang, R.~Fan, \emph{et~al.}, ``Calibrating parameters
  of power system stability models using advanced ensemble {K}alman filter,''
  \emph{IEEE Transactions on Power Systems}, vol.~33, no.~3, pp. 2895--2905,
  2018. [Online]. Available:
  \url{https://ieeexplore.ieee.org/document/8068956.}
\BIBentrySTDinterwordspacing

\bibitem{Hiskens2006Sensitivity}
\BIBentryALTinterwordspacing
I.~A. {Hiskens} and J.~{Alseddiqui}, ``Sensitivity, approximation, and
  uncertainty in power system dynamic simulation,'' \emph{IEEE Transactions on
  Power Systems}, vol.~21, no.~4, pp. 1808--1820, 2006. [Online]. Available:
  \url{https://ieeexplore.ieee.org/document/1717585}
\BIBentrySTDinterwordspacing

\bibitem{ODE}
T.~Ding and C.~Li, \emph{A Course on Ordinary Differential Equations}.\hskip
  1em plus 0.5em minus 0.4em\relax Beijing: Higher Education Press, 2004.

\bibitem{Gronwall}
\BIBentryALTinterwordspacing
S.~S. Dragomir, \emph{Some Gronwall type inequalities and applications}.\hskip
  1em plus 0.5em minus 0.4em\relax Nova Science, 2003. [Online]. Available:
  \url{https://papers.ssrn.com/sol3/papers.cfm?abstract\_id=3158353}
\BIBentrySTDinterwordspacing

\bibitem{Hu2004MatExp}
\BIBentryALTinterwordspacing
G.~Hu and M.~Liu, ``The weighted logarithmic matrix norm and bounds of the
  matrix exponential,'' \emph{Linear Algebra and its Applications}, vol. 390,
  pp. 145--154, 2004. [Online]. Available:
  \url{https://doi.org/10.1016/j.laa.2004.04.015.}
\BIBentrySTDinterwordspacing

\end{thebibliography}

\end{document}